\documentclass[letterpaper]{amsart}
\usepackage[english]{babel}
\usepackage[utf8]{inputenc}
\usepackage{amsmath}
\usepackage{graphicx}
\usepackage[colorinlistoftodos]{todonotes}
\usepackage{amssymb}
\usepackage{amscd} 		% provides CD environment for commutative diagrams (and exact seqs)
\usepackage{dsfont}
\usepackage{fixltx2e}
\usepackage{hyperref} 
\usepackage{cite}
\usepackage{tikz-cd}
\usepackage{listings}

\newtheorem{theo}{Theorem}[section]
\newtheorem*{theo*}{Theorem}

\newtheorem{algo}{Algorithm}

\newtheorem*{conj*}{Conjecture}
\newtheorem{coro}{Corollary}

\newtheorem{prop}[theo]{Proposition}
\newtheorem{rema}{Remark}
\theoremstyle{definition} \newtheorem{exam}{Example} \newtheorem*{defi}{Definition}

\newcommand{\CC}{\mathbb{C}}

\newcommand{\PP}{\mathbb{P}}

\newcommand{\ZZ}{\mathbb{Z}}

\newcommand{\cI}{\mathcal{I}}
\newcommand{\cO}{\mathcal{O}}

\newcommand{\cM}{\mathcal{M}}
\newcommand{\cT}{\mathcal{T}}

\def\and{\quad\mathrm{and}\quad}

\def\beq{\begin{equation}}
\def\eeq{\end{equation}}

\newcommand{\id}{\mathds{1}}

\title[]{Local Euler Obstructions of Reflective Projective Varieties}
 \author{Xiping Zhang}
 \email{xzhmath@gmail.com}
\date{\today}

\begin{document}
\maketitle

\begin{abstract}
In this note we introduce the concept of reflective projective varieties. These are stratified projective varieties with certain dimension constraints on their dual varieties. We prove that for such varieties, the Chern-Schwartz-MacPherson classes of the strata completely determine the local Euler obstructions and the polar degrees.  We also propose an algorithm to compute the local  Euler obstructions when such varieties are formed by group orbits. As examples we compute the local Euler obstructions of quadratic hypersurfaces and ordinary determinantal varieties to illustrate our method.
\end{abstract}
%\maketitle
%\tableofcontents
\section{Introduction}
The local geometric invariants around singular points are essential objects to study in algebraic geometry and singularity theory. Among all sorts of invariants, the local Euler obstructions are of central importance in the study of singular spaces. Defined as the obstruction to extend the distance $1$-form  after lifting to the Nash transform; 
it is the key ingredient in MacPherson's  
proof of the existence and uniqueness of Chern class on singular spaces, which was conjectured by Deligne and Grothendieck. For more about local Euler obstructions we refer to the nice papers \cite{B-NG} and \cite{B-NG2}. 
Such Chern class is equivalent to a previous construction of Schwartz (Cf. \cite{MR629125}\cite{MR35:3707}\cite{MR32:1727}), thus is called Chern-Schwartz-MacPherson class.

In \cite{BDK} the authors proved that the local Euler obstruction is equal to the local characteristics defined by Kashiwara in \cite{Kashiwara73}, in which he introduced the famous index theorem for holonomic D-modules. As the key ingredients in both singularity theory and Kashiwara's index theorem, the local Euler obstructions of stratified spaces have been deeply studied in singularity theory.  
In particular when such stratification is induced by a  group action, following the work of Kazhdan-Luzstig \cite{K-L80} the local Euler obstructions are closely related to the coefficients of the Kazhdan-Luzstig polynomials. Thus they are also intensely studied in geometric representation theory. 
Despite the great importance, the local Euler obstructions are in general very hard to compute.  
Many authors have been working on formulas that make the computation easier. We refer to  \cite{BLS} \cite{BFL} \cite{Gonzalez}    \cite{Rod-Wang2} \cite{M-R20}   \cite{Le-Tessier81} for many different approaches. 

In this note we consider a special type of projective varieties: the reflective projective varieties. These are stratified varieties with certain dimension constraints on the dual varieties of the strata. The precise definition is given in Definition~\ref{defi; reflectiveProjVar}. A common example of reflective projective varieties is the projectivizations of orbital closures of a $G$-representation $V$, where $G$ is a connected algebraic group acting on $V$ by finite orbits. We call such orbits reflective group orbits. 
We take advantage of Aluffi's duality result in \cite{Aluffi16}, which relates the Chern-Mather classes between dual varieties, to prove the main result of this paper: 
\begin{theo*}[Theorem\ref{theo; reflectiveProjVar}, Corollary\ref{coro; polar}]
Let $V$ be a vector space of dimension $N$. 
Let $X=\sqcup_{i=1}^n X_i\subset \PP(V)$ be a reflective projective variety, with its dual variety $Y=X^\vee$ stratified by $Y=\sqcup_{j=1}^s Y_j$. 
The Chern-Schwartz-MacPherson classes of the strata 
$$\{c_{sm}^{X_1}(H),\cdots ,c_{sm}^{X_n}(H); c_{sm}^{Y_1}(H), \cdots ,c_{sm}^{Y_s}(H) \}$$ 
completely determine the local Euler obstruction functions $Eu_X$ and $Eu_Y$, and thus the Chern-Mather classes $c_M^X(H)$, $c_M^Y(H)$, or equivalently the polar degrees $\{P_k(X), P_r(Y)|k,r\}$.  
\end{theo*}

As pointed out in \cite{FRW18}\cite{RV18}, for certain group orbits the   Chern-Schwartz-MacPherson classes are precisely the solutions of some constraint equations coming from localizations and specializations/restrictions of the group action. These solutions are certain weight functions   corresponding to the $K$-theoretic stable envelops studied in \cite{Maulik-Okounkov}.  This provides a way to obtain the Chern-Schwartz-MacPherson classes   directly from representation theory, besides the classical intersection method. This allows more flexible on computing in practice. 
Thus in Corollary~\ref{coro; ReflectiveOrbits} and Algorithm~\ref{OrbAlgorithm} we prove  an effective algorithm to compute the local Euler obstructions for reflective group orbits, using the Chern-Schwartz-MacPherson classes. As applications  we consider   quadratic hypersurfaces and (ordinary) determinantal varieties. They are natural examples of reflective projective varieties and reflective group orbits, and in Corollary~\ref{coro; quadratic} and Corollary~\ref{coro; detvar} we show that how to apply this duality method on effectively computing their local Euler obstructions. We also compute the Milnor  classes for quadratic hypersurfaces. 

The paper is organized as follows. In \S\ref{S; preliminary} we briefly recall the definitions and the tools we will use, namely the theory of Chern-Schwartz-MacPherson classes and Aluffi's projective duality. Then in \S\ref{S; ReflectiveVar} we state and prove the main result of the paper. In sections \S\ref{S; quadratic} and \S\ref{S; DerVar} we show that our method can effectively compute the local Euler obstructions for quadratic hypersurfaces and (ordinary) determinantal varieties.

\section{Preliminary}
\label{S; preliminary}
\subsection{Chern-Schwartz-MacPherson Classes}
Let $X\subset \PP^N$ be a projective variety. The group of constructible function is defined as the abelian group generated by indicator functions $\id_V$ for all irreducible subvarieties $V\subset X$. We define the pushforward for a proper morphism $f\colon X\to Y$ as follows. For any closed subvariety $V\subset X$, the pushforward $F(f)(\id_V)(y)$ evaluates $\chi(f^{-1}(y)\cap V)$ for any $y\in Y$. 
This makes $F$ a functor from projective complex varieties  to  the Abelian group category. The group $F(X)$ has $\{\id_V\}$ as a natural base.

When the base field is $\CC$, 
in 1974  MacPherson defined a local  measurement for  singularities and 
names it the local Euler obstruction. He
proved that the local Euler obstruction functions $\{Eu_V|\text{V is a subvariety of X}\}$ also form a base for $F(X)$. He also constructed a natural transformation $c_*$ by sending  the local Euler obstruction function $Eu_V$ to the Chern-Mather class $c_M(V)$. The precise definition is left to the next subsection. 
The following theorem answered a conjecture by Deligne and Grothendieck: 
\begin{theo}[\cite{MAC}]
The map $c_*$ defined above is the unique natural transform from $F$ to the homology functor $H_*$ satisfying the following normalization property: $c_*(\id_X)=c(TX)\cap [X]$ when $X$ is smooth.
\end{theo}
\begin{rema}
Another definition of Chern classes on singular varieties is due to M.-H. Schwartz, who uses obstruction theory and radial frames to construct such classes. Details of the construction can be found in~\cite{MR32:1727}~\cite{MR35:3707}~\cite{MR629125}. Also in~\cite{MR629125} it is shown that these classes correspond, by Alexander isomorphism, to the classes defined by MacPherson in the above theorem. 
\end{rema} 

The local Euler obstruction was originally defined via obstruction theory over $\CC$, as the obstruction to extend the lift of the distance $1$ form on the Nash transform. 
In \cite{Gonzalez} the author  gave an algebraic formula using the Nash blowup that works for  arbitrary algebraically closed field. An equivalent formula using conormal spaces was given by Sabbah in \cite{Sabbah}. 
In 1990 Kennedy modified Sabbah's Lagrangian intersections and proved the following generalization.
\begin{theo}[\cite{Kennedy}]
Replace the homology functor by the Chow functor, 
MacPherson's natural transformation extends to arbitrary algebraically closed field of characteristic $0$.
\end{theo}
\begin{rema}
Let $X$ be a complete variety with  constant map $k\colon X\to \{p\}$. Then the covariance property of $c_*$
shows that
\begin{align*}
\int_X c_{sm}(Y)
=&~ \int_{\{p\}} Afc_*(\id_Y)=\int_{\{p\}} c_*Ff(\id_Y)\\
=&~ \int_{\{p\}} \chi(Y)c_*(\id_{\{p\}})=\chi(Y). \qedhere
\end{align*}
This observation gives a generalization of the classical
Poincar\'e-Hopf Theorem to possibly singular varieties.
\end{rema}

\subsection{Local Euler Obstruction}
In this subsection we briefly review the definition of local Euler obstruction and Chern-Mather class, which play central role in the construction of $c_*$. The following algebraic definition is due to Gonz\'{a}lez-Sprinberg \cite{Gonzalez}.

Let $\iota\colon X\to \PP^N$ be a (closed) projective subvariety of dimension $d$, we define the Nash transform of $X$ to be 
\[
\hat{X}:=\text{closure of }\{(x,T_x X )|x \text{ is a smooth point} \}\subset G_d(T\PP^N) 
\]
The projection map $p\colon \hat{X}\to X$ is birational and is isomorphic over the smooth locus $X_{sm}$. The universal sub-bundle $S$ of $G_d(T\PP^N)$ restricts to a rank $d$ vector bundle on $\hat{X}$, denoted by $\cT_X$. 
\begin{defi}
The local Euler obstruction of $X$ is defined as follows: for any $x\in X$ we define
\[
Eu_X(x):=\int_{p^{-1}(x)} c(\cT_X)\cap s(p^{-1}(x), \hat{X}) \/.
\]
Here $s(A,B)$ is Fulton's Segre class  in \cite{INT}. The Chern-Mather class of $X$ is defined as
\[
c_M^X:= \iota_*p_*(c(\cT_X)\cap [\hat{X}])  \/.
\]
\end{defi}

Recall that the Chow group(ring) of $\PP^N$ is $\ZZ[H]/H^{N+1}$, where $H$ here is the hyperplane class $c_1(\cO(1))\cap [\PP^N]$.
\begin{defi}
The Chern-Schwartz-MacPherson class  of $X$, denoted by $c_{sm}^X(H)$  is defined as the pushforward   $\iota_*c_*(\id_X)$  in $A_*(\PP^N)$.
\end{defi}
Notice that when $X$ is smooth, the Chern-Mather class and the Chern-Schwartz-MacPherson class  all equal to the total Chern class $i_*(c(TX)\cap [X])$.

\begin{prop} 
We have the following properties.
\begin{enumerate}
\item  $Eu_X(x)$ is a local invariant, thus only depends on an open neighborhood of $x$ in $X$.
\item  If $x\notin X$, then $Eu_X(x)=0$.
\item  If $x\in X$ is a smooth point, then $Eu_X(x)=1$. But notice that $Eu_X(x)=1$ does NOT imply that $x$ is smooth (Cf. \cite{Nodland},\cite{M-R20} and \cite{Xiping3}).
\item  The Euler obstruction has the product property, i.e., $Eu_{X\times Y}(x\times y)=Eu_{X}(x)\times Eu_Y(y)$.
\end{enumerate}
\end{prop}

\begin{prop}
\label{prop; affineproj}
Let $X\subset \PP(V)$ be a projective variety, and let $\Sigma\subset V$ be its affine cone. Let $q\in V\setminus \{0\}$, and let $\bar{q}$ be its image in $\PP(V)$. Then we have
\[
Eu_X(\bar{q})=Eu_{\Sigma}(q) \/.
\]
\end{prop}
\begin{proof}
This is because that locally around $\bar{q}$, the affine cone $\Sigma$ is isomorphic to the product $X\times \CC^*$. This result follows from the product property mentioned above.
\end{proof}
\subsection{Involutions of Chern  Classes}
\label{S; projDual}
In this subsection we introduce the key ingredient we will use in the paper: the projective duality involution introduced by Aluffi in \cite{Aluffi16}.
Let $\ZZ_{\leq d}[x]$ be the vector space of degree $\leq d$ polynomials. For any $d\ge 0$ 
we define the following transformations  $\cI_d \colon \ZZ_{\leq d}[x]\to \ZZ_{\leq d}[x]$by
\begin{align*}
\cI_d \colon & f(x)\mapsto f(-1-x)-f(-1)((1+x)^{d+1}-x^{d+1})  
\end{align*}
\begin{prop}
\label{prop; involution}
The  transformations $\cI_d$ has the following properties:
\begin{enumerate}
\item  For any polynomial $f$ with no constant term,
$\cI_d(\cI_d(f))=f$. Thus $\cI_d$ is an involution on the set of polynomials.
\item The involution  $\cI_d$ is linear, i.e., 
$\cI_d (af+bg)=a\cI_d (f)+b\cI_d(g)$. 
%\item The involution  $\cI_d$ takes $Hf(H)$ to 
%\[
%\cI_d(Hf(H))=
%(-1-H)\cI_d(f)-f(-1)H((1+H)^{d+1}-H^{d+1}) \/.
%\]
\end{enumerate}
\end{prop}

Let $X\subset \PP(V)$ be a projective variety of dimension $n$, the dual variety $X^\vee$ is a projective variety in $\PP(V^*)$,
defined as 
\[
X^\vee:=
\overline{
\{
H\in \PP(V^*) |T_xX\subset H \text{ for some smooth point } x\in X_{sm} 
\}
}
\]
\begin{theo}[Projective Duality \cite{Aluffi16}]
\label{theo; involution}
The involution  $\cI_{\dim V-1}$  takes the signed Chern-Mather class of $X$ to the signed Chern-Mather class of the dual variety $X^\vee$. More precisely, we have 
\[
(-1)^{\dim X^\vee}c_M^{X^\vee}(H)=(-1)^{\dim X} \cI_{\dim V-1}( c_M^X (H) )\/.
\]
\end{theo}
For more details about projective dual varieties we refer to \cite{INT}\cite{Tevelev}. We will use the following result. 
\begin{prop}[\cite{Tevelev}, \S 2.2]
\label{prop; duality}
Let $G$ be a connected algebraic group, and 
let $V$ be a $G$-representation. Let $V^*$ be the dual representation of $G$. If there are only finitely orbits, and all the orbits are cones, then the projectivized orbits in $\PP(V)$ and  $\PP(V^*)$ are projective dual to each other,
\end{prop}

For the rest of the paper we consider projective varieties over algebraically closed field $k$ of characteristic $0$.
\section{Local Euler Obstructions of Reflective Projective Varieties}
\label{S; ReflectiveVar}
In this section we consider a special type of projective varieties: the reflective varieties.  
%For such varieties we propose an algorithm to compute their local Euler obstructions.
\subsection{Main Result}
\begin{defi}
\label{defi; reflectiveProjVar}
Let $V$ be a vector space of dimension $N$. 
We say a projective variety $X\subset \PP(V)$ is \textit{reflective} if $X$ admits a finite stratification $X=\cup_{i=1}^n X_i$, such that 
\begin{enumerate}
\item $X_n=\bar{X}_n\subset \bar{X}_{n-1}\subset \cdots \bar{X}_2 \subset \bar{X}_1=X$;
\item The dimension of the dual varieties $\bar{X}_i^\vee$ satisfy the following inequalities:
$
\dim X^\vee=\dim \bar{X}^\vee_0 <  \dim \bar{X}^\vee_1 < \cdots < \dim \bar{X}^\vee_n \/.
$
\item $Y=X^\vee=\cup_{j=1}^m Y_j$ form a finite stratification.
\end{enumerate}
\end{defi} 
\begin{exam}
Let $X_A\subset \PP(V)$ be a quadratic hypersurface defined by $x^tAx=0$ for some symmetric $n\times n$ matrix $A$.  The hypersurface $X_A$ is smooth if and only if $|A|\neq 0$, and the dual variety $X_A^\vee$ is also a smooth hypersurface.  When $A$ is singular of rank $k<n$,  as shown in \cite{GKZ},  the dual variety $X_A^\vee$ is a smooth quadratic hypersurface in $\PP^{r-1}\subset \PP(V)$.  The singularity locus $S_A$ of $X_A$ is given by the solution of the linear system $AX=0$, thus is a dimension $n-r-1$ linear space. The dual space of $S_A$ is isomorphic to $\PP^{r}$. This shows that the singular quadratic hypersurfaces are reflective. 
\end{exam}

\begin{theo}
\label{theo; reflectiveProjVar} 
We denote $I=I_{N-1}$ to be Aluffi's projective duality involution. 
Assume that there are not-all-zero integers $\{\alpha_i|i=1,\cdots, n\}$ and $\{\beta_j|j=0, \cdots, m\}$ such that the following involution equality holds. 
\[
I\left(\sum_{r=1}^n \alpha_r\cdot c_{sm}^{X_r}\right)=\sum_{s=1}^m  \beta_s \cdot c_{sm}^{Y_s}; \quad \alpha_1=\beta_{1}=1 \/.
\]
Then we have
\[
\alpha_i=Eu_{X}(X_i); \quad  \beta_i=Eu_{Y}(Y_i) \/.
\]
%The Theorem~\ref{theo; algorithm} and algorithm works for reflective  projective varieties.
\end{theo}
\begin{proof}
%In fact one can see that the proof of Proposition~\ref{theo; algorithm} only involves the duality assumption and the dimension counts. Thus the same proof works here.
We define $\hat{c}_M^{\bar{X}_i}(H):=(-1)^{\dim X_i} c_M^{\bar{X}_i}(H)$ to be the sighed Chern-Mather polynomials. 
Since the local Euler obstructions form a basis for the constructible function group, 
 there are unique integers $\{E_{r,i}\}$ and $\{D_{r,j}\}$ such that 
 $E_{r,r}=(-1)^{\dim X_r}$, $D_{s,s}=(-1)^{\dim Y_s}$ and
\[
c_{sm}^{X_r}(H)=  \sum_{i=r}^{n} E_{r,i}\cdot \hat{c}_M^{\bar{X}_i}(H); \quad 
c_{sm}^{Y_s}(H)=  \sum_{j=s}^{m} D_{s,j}\cdot \hat{c}_M^{\bar{Y}_j}(H) \/.
\]
The involution equality can then be written as
\begin{align*}
\sum_{j=1}^m \left( \sum_{s=1}^j  \beta_sD_{s,j} \right)  \hat{c}_M^{\bar{Y}_j} & =
\sum_{s=1}^m \sum_{j=s}^m  \beta_sD_{s,j}\cdot \hat{c}_M^{\bar{Y}_j} =
\sum_{s\geq m}  \beta_s \cdot c_{sm}^{Y_s}
\\
=& I\left(\sum_{r=1}^n \alpha_r\cdot c_{sm}^{X_r}\right) =
I\left(\sum_{r=1}^n \sum_{i=r}^{n} \alpha_rE_{r,i}\cdot \hat{c}_M^{\bar{X}_i} \right) \\
=&  \sum_{r=1}^n \sum_{i=r}^{n} \alpha_rE_{r,i}\cdot I(\hat{c}_M^{\bar{X}_i} ) 
=\sum_{r=1}^n \sum_{i=r}^{n} \alpha_rE_{r,i}\cdot \hat{c}_M^{\bar{X}^\vee_i} \\
=&  \sum_{i=1}^{n} \left( \sum_{r=1}^i \alpha_rE_{r,i}\right) \hat{c}_M^{\bar{X}^\vee_i} 
\end{align*}
Thus we obtain the following equality
\[
\sum_{j=1}^m \left( \sum_{s=1}^j  \beta_sD_{s,j} \right)  \hat{c}_M^{\bar{Y}_j}
= \sum_{i=1}^{n} \left( \sum_{r=1}^i  \alpha_rE_{r,i}\right) \hat{c}_M^{\bar{X}^\vee_i} \/.
\]
Recall that we have an increasing  chain of dimensions from the definition of reflective projective variety:
\[
\dim \bar{Y}_m <\dim \bar{Y}_{m-1} < \cdots < \dim \bar{Y}_1=  \dim \bar{X}^\vee_1 <  \dim \bar{X}^\vee_2 < \cdots < \dim \bar{X}^\vee_n 
\]
First we look at the highest dimension term $\bar{Y}_m$. Among all the polynomials $\{\hat{c}_M^{\bar{X}^\vee_i};\hat{c}_M^{\bar{Y}_j}\}$, the polynomial $\hat{c}_M^{\bar{Y}_m}$ is the only polynomial that contains $[\PP^{\dim Y_m}]$. Thus we have $\sum_{s=1}^m  \beta_s D_{s,m}=0$. 
Since the coefficient of $\hat{c}_M^{\bar{Y}_m}$ is $0$, then the polynomial $\hat{c}_M^{\bar{Y}_{m-1}}$ is the only polynomial that contains $[\PP^{\dim Y_{m-1}}]$. This forces $\sum_{s=1}^{m-1}  \beta_s D_{s,m-1}=0$.

Continue this process until we meet $ \bar{Y}_1=   \bar{X}^\vee_1$. At this level we obtain 
$\beta_1D_{1,1}=\beta_1=\alpha_1E_{1,1}=\alpha_1$. Then we repeat the above procedure on the dimension of $X_r$'s, eventually we obtain the following linear system:
\begin{align*}
\begin{cases}
\sum_{j=1}^{s} \beta_j D_{j,r}= 0 & s=2,3,\cdots , m\\
\beta_1D_{1,1}= \alpha_1 E_{1,1}=1 & \\
\sum_{i=1}^{r} \alpha_i E_{i,r}= 0 & r=2,3,\cdots , m 
\end{cases}
\end{align*}
Written into matrix form we have 
\begin{align*}
\begin{bmatrix}
D_{1,m} & D_{2,m}  & \cdots  & D_{m-1,m}  & D_{m,m} \\
D_{1,m-1}  & D_{2,2}  &  \cdots  & D_{m-1,m-1}   & 0 \\
\cdots  & \cdots  & \cdots  & \cdots  & \cdots  \\
D_{1,1}  & 0  & \cdots  & 0  & 0  
\end{bmatrix}
\begin{bmatrix}
\beta_1 \\ \beta_2 \\\cdots \\\beta_{m-1} \\\beta_m 
\end{bmatrix}
=& \begin{bmatrix}
0 \\ 0 \\\cdots \\0 \\ 1 
\end{bmatrix} \\
\begin{bmatrix}
E_{1,n} & E_{2,n}  & \cdots  & E_{n-1,n}  & E_{n,n} \\
E_{1,n-1}  & E_{2,2}  &  \cdots  & D_{n-1,n-1}   & 0 \\
\cdots  & \cdots  & \cdots  & \cdots  & \cdots  \\
E_{1,1}  & 0  & \cdots  & 0  & 0  
\end{bmatrix}
\begin{bmatrix}
\alpha_1 \\ \alpha_2 \\ \cdots \\ \alpha_{m-1} \\ \alpha_m 
\end{bmatrix}
=& \begin{bmatrix}
0 \\ 0 \\\cdots \\0 \\ 1 
\end{bmatrix}
\end{align*}
The two matrices are both of full rank, since that have diagonal entries $D_{s,s}=E_{r,r}=1$ for any $r,s$, thus the solution of $\{\alpha_1,\cdots ,\alpha_n, \beta_1, \cdots ,\beta_m\}$ is unique. Meanwhile, notice that $\{\alpha_i=Eu_{X}(X_i); \beta_j=Eu_{Y}(Y_j)\}$ is indeed a solution: the involution does interchanges the signed Chern-Mather classes between dual varieties. Thus by the uniqueness of the solution we completes the proof.
\end{proof}

This theorem shows that for reflective projective varieties, the Chern-Schwartz-MacPherson classes of the strata  completely determine the local Euler obstructions. Since the  Chern-Mather classes are linear sums of Chern-Schwartz-MacPherson classes with weights being local Euler obstructions, they also determine the Chern-Mather classes. Moreover, as shown in \cite{Piene15}, the coefficients of the Chern-Mather classes are equivalent to the polar degrees, i.e.,  degrees of the polar varieties. This gives us the following. 
\begin{coro}
\label{coro; polar}
For a reflective projective variety $X=\sqcup X_i$ with dual variety $X^\vee=Y=\sqcup Y_j$, the Chern-Schwartz-MacPherson classes of the strata $\{X_i\}$ and $\{Y_j\}$ completely determine the local Euler obstructions $Eu_X (X_i)$ and $Eu_{Y}(Y_j)$. Moreover, they completely determine the Chern-Mather classes $c_M^X$ and $c_M^Y$, and thus the  polar degrees $P_k(X)$ and $P_k(Y)$. 
\end{coro}
%Based on this we then propose the following algorithm:
%\begin{algo}[Algorithm to compute the Local Euler Obstructions]
%\label{algorithm}
%The following algorithm for  local Euler obstructions.
%\begin{enumerate}
%\item[Step 0] For each pairs of orbits $\cO_i$  and $\cO_j$, notice that  
%$Eu_{\bar{\cO}_i}(\cO_j)=Eu_{\bar{S}_i}(S_j)$ passes to the projective setting for $j\leq s$. 
%\item[Step 1] For each projectivized orbit $S_i$ we compute its Chern-Schwartz-MacPherson class $c_{sm}^{S_i}(H)$.
%\item[Step 2] For each projectivized orbit $S'_i$ we compute its Chern-Schwartz-MacPherson class $c_{sm}^{S'_i}(H)$.
%\item[Step 3] For each $r$ we set up the following linear system 
%\begin{align*}
%\sum_{k=r}^{s-1} x^r_k c_{sm}^{S_k}(H) -\sum_{k=s-r}^{s-1} y^r_{k} c_{sm}^{S'_k}(H)=0; x^r_r=y^{r}_{s-r}=1 \/. 
%\end{align*}
%This is a linear system since  $c_{sm}^{S_k}(H)$ 
%and $c_{sm}^{S'_k}(H)$ 
%are polynomials in $H$, and the equality gives at least $s+1$ linear equations concerning the coefficients of powers of $H$. As proved in the previous Proposition, the solution $\{x^r_k, y^r_{k}\}$ are the  local Euler obstructions of $\bar{\cO}_r$ and $\bar{\cO}'_{s-r}$ at each stratum.
%\item[Step 4]The local Euler obstruction  of $\bar{\cO}_r$ at $\cO_s=\{0\}$ come  from the algebraic Brasselet-Lê-Seade type formula proved in \cite{Xiping3}. 
%\[
%Eu_{\bar{\cO}_r}(0)=\sum_{k\geq r} (-1)^{\dim S_k}c_{sm}^{S_k}(-1)\cdot Eu_{\bar{\cO}_r}(\cO_k) \/.
%\]
%Same argument applies to  $Eu_{\bar{\cO}'_{s-r}}(0)$.
%\end{enumerate}
%\end{algo}

\subsection{On Reflective Group Orbits}
\label{S; ReflectiveOrbits}
A common example of reflective projective varieties comes from group actions.  
Let $G$ be a connected algebraic group and  $V$ be a linear $G$ representation of dimension $N$ with finite orbits $\cO_1,\cdots ,\cO_{s}$. From Proposition \ref{prop; duality} we can label the dual orbits by the duality correspondence: $\PP(\bar{\cO}_k)$ is dual to $\PP(\bar{\cO}'_{s-k})$.  
\begin{defi}
\label{assu; reflective}
We call such representation $V$ \textit{reflective} if the following assumption is satisfied. 
\begin{enumerate}
\item The $G$ action contains the scalar multiplication: the orbits are $k^*$ cones. 
\item For any $i$ we have $\bar{\cO}_i=\cup_{j\geq i}  \cO_j$ and $\bar{\cO}'_i=\cup_{j\geq i}  \cO'_j$. 
\item All the orbits $\cO_i$ and $\cO'_i$ are purely dimensional.
\end{enumerate}
The  assumptions forces  $\cO_s=\cO'_s=\{0\}$.
We call such group orbits \textit{Reflective Orbits}.
\end{defi}
 
\begin{coro}
\label{coro; ReflectiveOrbits}
For any $i$ we denote  $S_i=\PP(\cO_i)$ and $S'_i=\PP(\cO'_i)$ to be the projectivized orbits  in $\PP(V)$ and $\PP(V^*)$. Let $I=I_{N-1}$ be  Aluffi's projective duiality involution. Let $c_{sm}^{S_i}=c_{sm}^{S_i}(H)$ and $c_{sm}^{S'_i}=c_{sm}^{S'_i}(H)$ be the Chern-Schwartz-MacPherson polynomials in $H$.  
For any $r$, assume that there are not-all-zeros integers $\{\alpha^r_i|i=r,\cdots, s-1\}$ and $\{\beta^r_j|j=s-r, \cdots, s-1\}$ such that 
\[
I_{N-1}\left(\sum_{i\geq r} \alpha^r_i\cdot c_{sm}^{S_i}\right)=\sum_{k\geq s-r}  \beta^r_i \cdot c_{sm}^{S'_k}; \quad \alpha^r_r=\beta^r_{s-r}=1 \/.
\]
Then we have
\[
\alpha^r_i=Eu_{\bar{S}_r}(S_i); \quad  \beta^r_i=Eu_{\bar{S}'_{s-r}}(S'_i) \/.
\]
\end{coro}
\begin{proof}
Denote $m_i$ and $m'_j$ to be the sighed Chern-Mather class polynomials: $m_i:=(-1)^{\dim S_i} c_M^{\bar{S}_i}(H)$ and $m'_j:=(-1)^{\dim S'_j} c_M^{\bar{S}'_j}(H)$. For any $r$ there are unique integers $E_{r,i}$ and $E'_{s-r,j}$ such that $E_{r,r}=(-1)^{\dim S_r}$, $E'_{s-r,s-r}=(-1)^{\dim S'_{s-r}}$ and
\begin{align*}
c_{sm}^{S_r}(H)=& \sum_{i=r}^{s-1} E_{r,i}\cdot (-1)^{\dim S_i} c_M^{\bar{S}_i}(H)  =\sum_{i=r}^{s-1} E_{r,i} m_i;\\
c_{sm}^{S'_{s-r}}(H)=& \sum_{j=s-r}^{s-1} E'_{s-r,j}\cdot (-1)^{\dim S'_i} c_M^{\bar{S}'_i}(H)=\sum_{j=s-r}^{s-1} E'_{r,i} m'_j \/.
\end{align*}
Since $\bar{S}_k$ is dual to $\bar{S}'_{s-k}$, recall from Proposition~\ref{S; projDual} that the involution $I_{N-1}$ then takes $m_k$ to $m'_{s-k}$.
Thus the involution equality can be written as
\begin{align*}
& \sum_{k\geq s-r}  \beta^r_k \cdot c_{sm}^{S_k}(H)
= \sum_{k=s-r}^{s-1} \sum_{j=s-r}^k \beta^r_iE'_{j,k} m'_k  \\
=&I_{N-1}\left(\sum_{k\geq r} \alpha^r_k \cdot c_{sm}^{S_k}(H) \right)
= I_{N-1}\left(\sum_{k=r}^{s-1}  \alpha^r_k \cdot \sum_{i=k}^{s-1} E_{k,i} m_i \right) \\
=&\sum_{k=r}^{s-1}  \alpha^r_k \cdot \sum_{i=k}^{s-1} E_{k,i} m'_{s-i} 
=  \sum_{k=1}^{s-r } \sum_{i=r}^{s-k} \alpha^r_iE_{i,s-k} m'_k  
\end{align*}
Since $\alpha^r_r=\beta^r_{s-r}=1$, we have
\[
m'_{s-r}+\sum_{k=s-r+1}^{s} \sum_{i=s-r}^k \beta^r_jE'_{j,k} m'_k 
=
\sum_{k=1}^{s-r-1} \sum_{i=r}^{s-k} \alpha^r_iE_{i,s-k} m'_k + m'_{s-r} \/.
\]
The corollary then follows from the same argument used in the  proof of Theorem~\ref{theo; reflectiveProjVar}.
\end{proof}

An advantage of reflective group orbits is that,
%the strata $Y_j$ in Definition~\ref{defi; reflectiveProjVar} are exactly the $X'_j$. 
%we can skip Setp ?? in Algorithm~\ref{VarAlgorithm}, which in general could be vary complicated. 
 as pointed out in \cite{FRW18}\cite{RV18}, besides the  classical approach from intersection theory  the Chern-Schwartz-MacPherson classes of such group orbits  can also be obtained from representation theory. This provides more choices in practice. 
In this case we have the following algorithm to compute the local Euler obstructions:
\begin{algo}[Algorithm for Reflective Group Orbits]
\label{OrbAlgorithm}
The following algorithm computed the  local Euler obstructions for Reflective Group Orbits.
\begin{enumerate}
\item[Step 0] For each pairs of orbits $\cO_i$  and $\cO_j$, notice that  
$Eu_{\bar{\cO}_i}(\cO_j)=Eu_{\bar{S}_i}(S_j)$ passes to the projective setting for $j\leq s$. 
\item[Step 1] For each projectivized orbit $S_i$ we compute its Chern-Schwartz-MacPherson class $c_{sm}^{S_i}(H)$.
\item[Step 2] For each projectivized orbit $S'_i$ we compute its Chern-Schwartz-MacPherson class $c_{sm}^{S'_i}(H)$.
\item[Step 3] For each $r$ we set up the following linear system 
\begin{align*}
\sum_{k=r}^{s-1} x^r_k c_{sm}^{S_k}(H) -\sum_{k=s-r}^{s-1} y^r_{k} c_{sm}^{S'_k}(H)=0; x^r_r=y^{r}_{s-r}=1 \/. 
\end{align*}
This is a linear system since  $c_{sm}^{S_k}(H)$ 
and $c_{sm}^{S'_k}(H)$ 
are polynomials in $H$, and the equality gives at least $s+1$ linear equations concerning the coefficients of powers of $H$. As proved in the previous Proposition, the solution $\{x^r_k, y^r_{k}\}$ are the  local Euler obstructions of $\bar{\cO}_r$ and $\bar{\cO}'_{s-r}$ at each stratum.
\item[Step 4]The local Euler obstruction  of $\bar{\cO}_r$ at $\cO_s=\{0\}$ come  from the algebraic Brasselet-Lê-Seade type formula proved in \cite{Xiping3}. 
\[
Eu_{\bar{\cO}_r}(0)=\sum_{k\geq r} (-1)^{\dim S_k}c_{sm}^{S_k}(-1)\cdot Eu_{\bar{\cO}_r}(\cO_k) \/.
\]
Same argument applies to  $Eu_{\bar{\cO}'_{s-r}}(0)$.
\end{enumerate}
\end{algo}

\begin{rema}
\label{rema; FRW}
More precisely, in \cite{FRW18}\cite{RV18} they proved that the equivariant Chern-Schwartz-MacPherson classes are characterized by certain Axioms. The equivariant classes are expressed as polynomials in the weights of the $G$ action, 
the axioms then provide constraint equations  the  Chern polynomials have to satisfy by studying localizations and specializations/restrictions. The solutions are certain weight functions from representation theory. 
In fact, in \cite{FRW18} they proved that for such orbits the solutions correspond to the $K$-theoretic stable envelops studied in \cite{Maulik-Okounkov}. Thus this algorithm provides a way to study the singularities of orbits closures directly from representation theory. 
\end{rema}

\begin{exam}
We use the following example to illustrate the algorithm.
We consider the symmetric rank stratification $M^S_3=\cup_{i=0}^3 \Sigma^{S \circ}_{3,i}$. From the computation in \cite{Xiping4} we have
\begin{align*}
c_{sm}^{\tau^{S \circ}_{3,0}}
=& 3H^4+6H^3+6H^2+3H+1   \\
c_{sm}^{\tau^{S \circ}_{3,1}}
=& 3H^5+6H^4+10H^3+9H^2  + 3H; \quad c_{sm}^{\tau^{S \circ}_{3,1}}(-1)=-1 \\
c_{sm}^{\tau^{S}_{3,2}}
=&  3H^5  + 6H^4  + 4H^3 ; \quad c_{sm}^{\tau^{S}_{3,2}}(-1)=-1
\end{align*}
The orbits $\tau^S_{3,1}$ is dual to $\tau^S_{3,2}$, with dimension $7$ and $4$ respectively. 
Then we have
\[
c_{sm}^{\tau^{S \circ}_{3,1}}(-1-H)+((1+H)^6-H^6)
+x\cdot c_{sm}^{\tau^{S}_{3,2}}(-1-H)+x\cdot((1+H)^6-H^6)=c_{sm}^{\tau^{S}_{3,2}}(H) \/.
\]
Expand the polynomials we get
\[
3H^5+6H^4+4H^3+x\cdot (3 H^5+6 H^4+10 H^3+9 H^2+3 H)=3H^5+6H^4+4H^3 \/.
\]
There is a unique solution $x=0$. This shows that $Eu_{\Sigma^{S }_{3,1}}(\Sigma^{S \circ}_{3,2})=0$. The Brasselet-Lê-Seade  formula then gives 
\[
Eu_{\Sigma^{S }_{3,1}}(0)=Eu_{\Sigma^{S }_{3,1}}(\Sigma^{S \circ}_{3,1})\cdot c_{sm}^{\tau^{S \circ}_{3,1}}(-1)+Eu_{\Sigma^{S }_{3,1}}(\Sigma^{S \circ}_{3,2})\cdot c_{sm}^{\tau^{S}_{3,2}}(-1)
=(-1)\cdot (-1)+ (-1)\cdot 0=1 \/.
\]
\end{exam}

\section{Application I: Quadratic Hypersurfaces}
\label{S; quadratic}
Let $X_A\subset \PP(V)$ be a quadratic hypersurface defined by $x^tAx=0$ for some symmetric $n+1\times n+1$ matrix $A$. Here $\dim V=n+1$. The hypersurface $X_A\subset \PP(V^*)$ is smooth if and only if $|A|\neq 0$, and the dual variety $X_A^\vee$ is also a smooth hypersurface.  When $A$ is singular of rank $r\leq n$,  the singular locus $S_A$ is the projective subspace $\PP(K)=\PP^{n-r}$, where $K=\ker A\subset V$. 

When $r=1$, the defining equation of $X_A$ is given by $l^2$ for some linear form $l$. Thus the reduced structure of $X_A$ is  smooth: it is a hyperplane. When $r=2$, the defining equation of $X_A$ is given by $l_1\cdot l_2$ for two distinct linear forms $l_1,l_2$. Thus the hypersurface $X_A$ is the union of two hyperplanes intersecting transversely at $\PP^{n-2}$. After taking normal slice this case is the nothing but the line arrangement $V(xy=0)\subset \PP^2$.   
So we will assume $r\geq 3$ in this section.

Let $K^{\perp}\subset V^*$ be the space of linear forms that vanish on $K$.
As shown in \cite{GKZ},  the dual variety $X_A^\vee$ is a smooth 
quadratic hypersurface in the linear subspace $\PP(K^{\perp})=\PP^{r-1}$.  The dual variety $S_A^\vee$ is isomorphic to the linear subspace $\PP^{r-1}$. 

\begin{rema}
In general, let  $A$ be a $m\times n$ matrix with $m\geq n$, we may consider the hypersurface in $\PP^{m+n-1}$ of the form $X^tAY=0$. This hypersurface is exactly the  hypersurface $X_B$, where $B$ is the $(m+n)\times (m+n)$ symmetric matrix
$
\begin{bmatrix}
0 & A \\
A^t & 0
\end{bmatrix}
\/.
$
\end{rema}
 
\begin{coro}
\label{coro; quadratic}
Let $r=r(A)$ be the rank of $A$. We assume that $r\geq 3$. 
\begin{enumerate}
\item Let $\Sigma_A\subset V$ be the affine cone of $X_A$. 
The Euler characteristic of the complex link space of $\Sigma_A$ at $0$ is $-2$.
\item The  Milnor number of $X_A$ at any point of $S_A$ equals $(-1)^{n+r}$. 
\item The Milnor class of $X_A$ is given by 
\[
\cM(X_A)=(-1)^{n+r}\sum_{k=0}^{n-r} \left(
\sum_{j=0}^k \binom{n-r+1}{k-j} (-2)^j
\right)\cdot H^{k+r}
\]
\item We denote $X_A^\circ:=X_A\setminus S_A$. The local Euler obstruction  of $X_A$ then equals:
$$Eu_{X_A}=\id_{X_A^\circ}+ ((-1)^r+1) \id_{S_A} \/.$$  
\end{enumerate}
\end{coro}
\begin{proof}
The Milnor class of $X_A$ can be uniquely written as 
\begin{equation}
\label{eq; Milnor}
\cM(X_A) =(-1)^{n-1}\left(\frac{2H((1+H)^{n+1}-H^{n+1})}{1+2H}-c_{sm}^{X_A}(H) \right)= \frac{\mu\cdot c_{sm}^{S_A}}{1+2H} \/.
\end{equation}
Let $e=Eu_{X_A}(S_A)$ be the local Euler obstruction of $X_A$ at $S_A$, and we can uniquely write 
\[
c_{sm}^{X_A}=c_M^{X_A}- (e-1)c_M^{S_A}=c_M^{X_A}- (e-1)c_{sm}^{S_A} \/.
\]
We denote $\alpha :=e-1$, $\lambda:=c_{sm}^{X_A}(-1)$ and $B_n:=(1+H)^{n+1}-H^{n+1}$.
Then we consider the following polynomial:
\[
P:= (-1)^{n-1}\left( 
2HB_n-(1+2H)(c_M^{X_A}- \alpha \cdot c_{M}^{S_A})\right)-\mu c_{M}^{S_A} \/.
\]
From (\ref{eq; Milnor}) we know that $P=0$. Let $\cI_n$ be the duality involution,  then we have $\cI_n(P)=0$.

On the other hand, we can directly compute $I_n(P)$. From Proposition~\ref{prop; involution} we have:
\[
\cI_n\left(
2H(c_M^{X_A}- \alpha \cdot c_{M}^{S_A})
\right)
= (-2-2H)\cI_n(c_M^{X_A}- \alpha \cdot c_{M}^{S_A})-\lambda HB_n \/.
\]
Meanwhile, by a direct computation we have
$
\cI_n(HB_n)=(-1)^{n+1} HB_n \/.
$
Thus we have the following
\begin{align*}
\cI_n(P)=& (-1)^{n-1}  
\cI_n(2HB_n)-\cI_n(c_M^{X_A}- \alpha  c_{M}^{S_A})-\cI_n(2H(c_M^{X_A}- \alpha  c_{M}^{S_A}))
-\cI_n(\mu c_{M}^{S_A}) \\
=& 2HB_n -(-1)^{n-1}\left( (-1)^{n+r-1} c_M^{X^\vee_A}- (-1)^{n-1} \alpha  c_{M}^{S^\vee_A} \right)+ (-1)^{n-1}\lambda HB_n \\
+& (-1)^{n-1} (2+2H)  \left( (-1)^{n+r-1} c_M^{X^\vee_A}- (-1)^{n-1} \alpha  c_{M}^{S^\vee_A} \right) 
-(-1)^{n-1}\mu c_{M}^{S^\vee_A} \\
=& (2+(-1)^{n-1}\lambda)HB_n +(-1)^r(1+2H)c_M^{X^\vee_A}-(1+2H) \alpha  c_{M}^{S^\vee_A} -(-1)^{n-1}\mu c_{M}^{S^\vee_A} \
=& 0
\end{align*}
Notice that $X_A^\vee$ is a smooth 
quadratic hypersurface in  $ \PP^{r-1}$, and  $S_A^\vee$ is isomorphic to  $\PP^{r-1}$. Thus we have
\[
c_M(X_A^\vee)=c_{sm}^{X_A^\vee}=\frac{2H^{n-r+2}(1+H)^r}{1+2H}; \quad
c_M(S_A^\vee)=c_{sm}^{S_A^\vee}=H^{n-r+1}(1+H)^r \/.
\]
Thus we have
\begin{align*}
0=& (2+(-1)^{n-1}\lambda)H\left((1+H)^{n+1}-H^{n+1}\right)-(-1)^{n-1}\mu H^{n-r+1}(1+H)^r \\
 & +(-1)^r 2H^{n-r+2}(1+H)^r-(1+2H) \alpha H^{n-r+1}(1+H)^r  \/.
\end{align*}
This gives the solution 
\[
\lambda=(-1)^n\cdot 2; \quad e=\alpha+1=(-1)^r+1 ; \quad \mu= (-1)^{n+r} \/.
\]
\end{proof}

\section{Application II: Generic Determinantal Varieties}
\label{S; DerVar}
The (generic) determinantal varieties are fundamental examples of non-isolated singularities, and are of great interest in singularity theory. For ordinary determinantal varieties the local Euler obtructions were computed in \cite{NG-TG} using topological method, and in \cite{Xiping2} for arbitrary algebraically closed field using intersection theory. In this section we use Corollary~\ref{coro; ReflectiveOrbits}  to  give a very short proof, from an  observation on the Chern-Schwartz-MacPherson class formula. 

\subsubsection{$q$-Polynomial of Determinantal Varieties}
Let $K$ be a characteristic $0$ algebraically closed field. 
Let $M_{m,n}=Hom(K^n,K^n)$  be the space of $n\times n$ ordinary matrices . 
The group  $GL_n(K)\times GL_n(K)$ acts on $M_{n,n}$ by $(P,X,Q)\mapsto PXQ^{-1}$. The orbits consist of matrices of fixed rank. For $0\leq k\leq n-1$ we denote them by 
\[
\Sigma^\circ_{n,k}:=\{X|\dim \ker X = k\}; \quad \Sigma_{n,k}:=\{X|\dim \ker X \geq k\} \/.
\]
They are algebraic varieties of dimension $n^2-k^2$. 
We denote the projections $\PP(\Sigma^\circ_{n,k})$ and $\PP(\Sigma_{n,k})$ by $\tau^\circ_{n,k}$ and $\tau_{n,k}$ respectively. 

As shown in \cite{Xiping4}, we define the  $Q$ classes in the Grassmannian $G(r,n)$:
\[
q_{n,r}(d)
:=  \left( \sum_{k=0}^{n(n-r)} (1+d)^{n(n-r)-k}c_k(Q^{\vee n})  \right) \left(\sum_{k=0}^{nr} d^{nr-k}c_k(S^{\vee n}) \right) \/;  
\]
and the  $q$ polynomial in variable $q_{n,r}(H)$ by setting
\[
q_{n,r}(d)
:=  \int_{G(r,n)} c(S^\vee\otimes Q)\cdot q_{n,r}(d)\cap [G(r,n)] - d^{n^2} \binom{n}{r} 
\] 
for integers $d$. 
The Chern-Schwartz-MacPherson classes $c_{sm}^{\tau^\circ_{n,k}}$ are then given by 
\[
c_{sm}^{\tau_{n,k}^{\circ}}(H)
= \sum_{r=k}^{n-1} (-1)^{r-k}\binom{r}{k} \cdot q_{n,r}(H) \/.
\]
\subsubsection{Local Euler Obstruction of Determinantal Varieties}
Despite the complicated form, as we will show later,  the $q_{n,r}(d)$, enjoys very nice symmetric property: it is symmetric with the $d\mapsto -1-d$ substitution. This induces the following interesting application:
\begin{coro}
\label{coro; detvar}
The local Euler obstruction function of determinantal varieties are given by
$$Eu_{\Sigma_{n,k}}(\Sigma^\circ_{n,r})=\binom{r}{k} \/.
$$
Thus the polynomials $q_{n,r}(H)$ equal  the Chern-Mather classes: 
$
q_{n,r}(H)=c_M^{\tau_{n,k}}(H) \/.
$
\end{coro}
\begin{proof}
Substituting $d$ by $-1-d$ in the expression of $q_{n,r}(d)$ we have the following:
\begin{small}
\begin{align*}
& q_{n,r}(-1-d)+(-1-d)^{n^2} \binom{n}{r} \\
&= \int_{G(r,n)} c(S^\vee\otimes Q)  \left(\sum_{k=0}^{n(n-r)} (-d)^{n(n-r)-k}c_k(Q^{\vee n})  \right) \left(\sum_{k=0}^{nr} (-1-d)^{nr-k}c_k(S^{\vee n}) \right)   \\
&= \int_{G(n-r,n)} c(S^\vee\otimes Q)  \left(\sum_{k=0}^{n(n-r)} (-d)^{n(n-r)-k}c_k(S^{ n})  \right) \left(\sum_{k=0}^{nr} (-1-d)^{nr-k}c_k(Q^{ n}) \right)    \\
&= (-1)^{n^2} \int_{G(n-r,n)} c(S^\vee\otimes Q)  \left(\sum_{k=0}^{n(n-r)} d^{n(n-r)-k}c_k(S^{\vee n})  \right) \left(\sum_{k=0}^{nr} ( 1+d)^{nr-k}c_k(Q^{\vee n}) \right)   \\
=& (-1)^{n^2} q_{n,n-r}(d)+ (-1)^{n^2} d^{n^2} \binom{n}{n-r}  
\end{align*}
\end{small}
Recall that the dimension of $\tau_{n,r}$ is $n^2-r^2-1$. Since above computation holds for any integer $d$, thus as polynomial we have
\[
\hat{q}_{n,r}(-1-H)-\hat{q}_{n,r}(-1)  \left((1+H)^{n^2}-H^{n^2}  \right )
= \hat{q}_{n,n-r}(H) \/.
\]
Here $\hat{q}_{n,r}(H)=(-1)^{\dim \tau_{n,r}}\cdot q_{n,r}(H)$  is the signed polynomial. 
Recall from \S\ref{S; projDual} that the projective duality 
 operation $I:=I_{n^2-1}$ takes the signed Chern-Mather class of $\tau_{n,r}$ to the signed Chern-Mather class of $\tau_{n,n-r}$: the dual variety of $\tau_{n,k}$ is exactly $\tau_{n,n-k}$. 
Let $c_k:=(-1)^{\dim \tau_{n,k}} \cdot c_M^{\tau_{n,k}}$ be the Chern-Mather polynomials, 
 this says that
$I(c_i)=c_{n-i}$ for $i=1, \cdots ,n-1$.

Proved in \cite{Xiping4}, the $q$ polynomial $q_{n,r}(H)$ equals $c_*(\phi_{n,r})$ for some constructible function $\phi_{n,r}\in F(\tau_{n,r})$, defined by $\phi_{n,r}(p)=\binom{r}{l}$ for $p\in \tau^\circ_{n,l}$. $($In fact, $\phi_{n,r}$ is the pushforward $p_*(\id_{\hat{\tau}_{n,r}})$, where $p\colon \hat{\tau}_{n,r}\to \tau_{n,r}$ is the Tjurina transform.$)$
Since the local Euler obstructions is a basis for $F(\tau_{n,r})$, there are unique integers $\beta^r_k$ such that
$q_{n,r}(H)=\sum_{k=r}^{n-1} \beta^r_k c_M^{\tau_{n,k}}(H)$, $\beta_r^r=1$. Then the involution equality shows that 
\[
I(\hat{q}_{n,r}(H))=I(\sum_{k=r}^{n-1} \beta^r_k c_k )=\sum_{k=r}^{n-1} \beta^r_k c_{n-k} =\hat{q}_{n,n-r}(H)=\sum_{k=n-r}^{n-1} \beta^r_k c_{k} \/.
\]  
Since the lowest degree terms all the polynomials $c_i$ have different degrees, 
expand the above equalities for $r=1,2,\cdots ,n-1$ one can see that there is a unique solution
\begin{align*}
\beta_B^A=&
\begin{cases}
1 & A=B \\
0 & A\neq B 
\end{cases}
\end{align*}
This shows that $\phi_{n,r}=Eu_{\tau_{n,r}}$, and thus
$q_{n,r}(H)=c_M^{\tau_{n,r}}$ is the Chern-Mather polynomial.
\end{proof}

\begin{rema}
However we fail to obtain such an easy proof for the Chern-Mather classes of symmetric and skew-symmetric determinantal varieties. This is due to the term
$\left(\sum_{k=0}^{e} d^{e-k} c_k(S^\vee\otimes Q^\vee) \right)$, when substituting $d$ by $-1-d$ we lose the symmetry. The fact that the involution still interchanges the signed Chern-Mather classes indicates that there should be more symmetry and vanishing patterns hide behind the integrations. 
It would be interesting to use Schubert calculus to give a similar  proof  to the symmetric and skew-symmetric determinantal varieties.
\end{rema}

\bibliographystyle{plain}
\bibliography{ref}
\end{document}